\documentclass[10pt,reqno]{amsart}
\usepackage{color}

\newtheorem{theorem}{Theorem}

\newtheorem{proposition}[theorem]{Proposition}

\newtheorem{remark}{Remark}

\newfont{\bb}{msbm10 at 12pt}

\def\bg{\overline{g}}
\def\wbg{\widetilde{\overline{g}}}

\def\<{\langle}     
\def\>{\rangle}

\newcommand{\bal}{\begin{align}}      \newcommand{\eal}{\end{align}}
\newcommand{\ba}{\begin{array}}      \newcommand{\ea}{\end{array}}
\newcommand{\bc}{\begin{center}}     \newcommand{\ec}{\end{center}}
\newcommand{\be}{\begin{enumerate}}  \newcommand{\ee}{\end{enumerate}}
\newcommand{\beQ}{\begin{eqnarray*}} \newcommand{\eeQ}{\end{eqnarray*}}
\newcommand{\bi}{\begin{itemize}}    \newcommand{\ei}{\end{itemize}}
\newcommand{\bt}{\begin{tabular}}    \newcommand{\et}{\end{tabular}}
\newcommand{\bdm}{\begin{displaymath}} \newcommand{\edm}{\end{displaymath}}

\newcommand{\ve}{\varepsilon}

\def\qed{\hfill{q.e.d.}\smallskip\smallskip}

\begin{document}
\title[]{A remark on the rigidity of conformally compact Poincar\'e-Einstein manifolds}

\author{Simon Raulot}
\address{Laboratoire de Math\'ematiques R. Salem
UMR $6085$ CNRS-Universit\'e de Rouen
Avenue de l'Universit\'e, BP.$12$
Technop\^ole du Madrillet
$76801$ Saint-\'Etienne-du-Rouvray, France.}
\email{simon.raulot@univ-rouen.fr}

\begin{abstract}
In this paper, we give an optimal inequality relating the relative Yamabe invariant of a certain compactification of a conformally compact Poin\-car\'e-Ein\-stein manifold with the Yamabe invariant of its boundary at infinity. As an application, we obtain an elementary proof of the rigidity of the hyperbolic space as the only conformally compact Poincar\'e-Einstein manifold with the round sphere as its conformal infinity.
\end{abstract}


\maketitle 
\pagenumbering{arabic}
\date{\today}   
 
%%%%%%%%%%%%%%%%%%%%%%%%%%%%%%%%%%%%%%%%%%%%%%%%%%%%%%%%%%%%%%%%%%%%%%

\section{Introduction}

%%%%%%%%%%%%%%%%%%%%%%%%%%%%%%%%%%%%%%%%%%%%%%%%%%%%%%%%%%%%%%%%%%%%%%

A conformally compact manifold $(X^n,g_+)$ is a connected complete Riemannian manifold whose metric extends conformally to a compact manifold $(\overline{X}^n,\bg)$ with boundary $\Sigma^{n-1}:=\partial\overline{X}^n$ whose interior is the original manifold $X^n$. So, through this extended conformal metric, the corresponding original metric determines a conformal structure $[\gamma]:=[\bg_{|\Sigma}]$ on the boundary, which is usually called the conformal infinity or the boundary at infinity. A particularly interesting class of conformally compact manifolds are the Poincar{\'e}-Einstein spaces, characterized by the fact that the manifold is Einstein which means that, up to a rescaling of the metric, its Ricci tensor satisfies the Einstein equation 
\begin{eqnarray}\label{EinsteinCond}
Ric_{g_+}=-(n-1)g_+.
\end{eqnarray}

In the context of Poincar\'e-Einstein manifolds, Andersson and 
Dahl \cite{AD} proved that, if the manifold is spin and its conformal infinity is an $(n-1)$-dimensional round sphere, then its only possible conformal compactification is the $n$-dimensional hyperbolic space. Their method relies on some generalization of the positive mass theorem in the context of asymptotically hyperbolic metrics. Since then, this approach leads to several generalizations of this result (see \cite{ChH,W} for example) in the context of spin manifolds. 

Without the spin assumption, the first significant result were obtained by 
Qing \cite{Q}. In this work, he used the existence of certain positive eigenfunctions of the Laplacian in order to build a complete non-compact asymptotically Euclidean Riemannian manifold with boundary, with non-negative integrable scalar curvature. Then he applied the positive mass theorem to the double manifold obtained by gluing through their boundaries two copies of such manifolds with boundary, proving that its ADM mass is zero. In this way, he could drop the spin assumption for dimensions between $3$ and $6$.

More recently, the rigidity of the hyperbolic as the only conformally compact Poincar\'e-Einstein manifold with the round sphere as its conformal infinity has been proved for more general manifolds and without the spin assumption. A rigorous proof of this fact appears recently in the work of Li, Qing and Shi \cite{LQS}. In fact, based on the idea in \cite{DJ} they manage to give a complete proof of a {\it relative volume inequality} on conformally compact manifolds which implies, among other things, the rigidity of the hyperbolic space. For Poincar\'e-Einstein manifolds, this result reads as
\begin{theorem}\label{RigHyp}
Suppose that $(X^n,g_+)$ is a conformally compact Einstein manifold with the round sphere as its conformal infinity. Then $(X^n,g_+)$ has to be the hyperbolic space. 
\end{theorem}

In this note, we use the recent approach of Gursky and Hang \cite{GH} to relate the Yamabe invariant of a certain compactification of the Poincar\'e-Einstein manifold with the Yamabe invariant of the boundary for the induced conformal class.  In fact, Gursky and Hang use the resolution of a certain Yamabe problem on $\overline{X}^n$ to prove an inequality which relates these two conformal invariant. The Yamabe-type problem involved in their paper was first study by Escobar \cite{E3} and consists to prove on a compact manifold with boundary, the existence, of a smooth conformally related metric with constant scalar curvature and minimal boundary. From their results, they deduce the existence of infinitely many positive conformal classes on the seven dimensional sphere which cannot be the conformal infinity of a Poincar\'e-Einstein metric on the eight dimensional ball. 

Here we use an other Yamabe-type problem (see Section \ref{CherrierEscobar} for the definitions) and obtain:
\begin{theorem}\label{main}
Let $(X,g_+)$ be a Poincar\'e-Einstein manifold of class $C^2$ with a smooth representative in its conformal infinity $\Sigma$. Assume also that the Yamabe invariant $Y(\Sigma,[\gamma])$ is non-negative. If the dimension $n\geq 4$, then 
\begin{eqnarray*}
Y(\overline{X},\Sigma,[\bg])^2\geq 4\frac{n-1}{n-2}Y(\Sigma,[\gamma]).
\end{eqnarray*}
If $n\geq 3$, then
\begin{eqnarray*}
Y(\overline{X},\Sigma,[\bg])^2\geq 32\pi\chi(\Sigma).
\end{eqnarray*}
Equality occurs if and only if $(X^n,g_+)$ is isometric to the standard hyperbolic space. 
\end{theorem}

As a direct consequence we obtain a new, simple and direct proof of Theorem \ref{RigHyp}. \\

Finally, we want to emphasize that, during the redaction of this work, it was pointed out that similar results were obtained independently by Chen, Lai and Wang in \cite{CLW}. 

%%%%%%%%%%%%%%%%%%%%%%%%%%%%%%%%%%%%%%%%%%%%%%%%%%%%%%%%%%%%%%%%

\section{Poincar\'e-Einstein manifolds}

%%%%%%%%%%%%%%%%%%%%%%%%%%%%%%%%%%%%%%%%%%%%%%%%%%%%%%%%%%%%%%%%

In this section, we collect basic facts on conformally compact manifolds. For more details on this subject, the reader could consult \cite{An,AH,AD,ChH,LQS,Q,W}. 

Let $X^n$ be the interior of a smooth compact manifold $\overline{X}^n$ with boundary $\Sigma=\partial\overline{X}^n$. A Riemannian metric $g_+$ is said to be conformally compact of class $C^{k,\alpha}$ if, for a smooth defining function $\rho$ for the boundary $\Sigma$ in $\overline{X}^n$, the metric $\bg=\rho^2g_+$ can be extended to a $C^{k,\alpha}$ Riemannian metric on $\overline{X}^n$. Recall that a smooth defining function $\rho$ for the boundary $\Sigma$ in $\overline{X}^n$ is a smooth nonnegative function from $\overline{X}$ such that 
\begin{itemize}
\item $\rho>0$ on the interior $X^n$;
\item $\rho=0$ on the boundary $\Sigma^{n-1}$;
\item $d\rho\neq 0$ on the boundary $\Sigma^{n-1}$. 
\end{itemize}

The compactification $\bg$ induces a metric $\gamma$ on the boundary $\Sigma$ which changes conformally when the defining function $\rho$ changes. Then a conformally compact metric $g_+$ always induces a conformal structure $[\gamma]$ on the boundary $\Sigma$ and the conformal manifold $(\Sigma,[\gamma])$ is called the conformal infinity of the conformally compact manifold $(X^n,g_+)$.

In the following we will always assume that the metric $g_+$ is an Einstein metric that is its Ricci tensor satisfies (\ref{EinsteinCond}). Such manifolds are called Poincar\'e-Einstein manifolds. It is very instructive for the following to remark that the Ricci tensor of $g_+$ and the Ricci tensor of any compactification $\bg=\rho^2g_+$ are related by 
\begin{eqnarray*}
Ric_{\bg}=-(n-2)\rho^{-1}\nabla_{\bg}^2\rho+\Big((n-1)\rho^{-2}(|\nabla^{\bg}\rho|^2-1)-\rho^{-1}\Delta_{\bg}\rho\Big)\bg
\end{eqnarray*}
since $g_+$ and $\bg$ are conformal (see \cite{Be} for example). Among other things, we can first deduce that any defining function $\rho$ of $\Sigma$ in $X$ satisfies $|\nabla^{\bg}\rho|_{\bg|\Sigma}=1$. Then since the second fundamental form of $\Sigma$ in $(\overline{X}^n,\bg)$ is $\nabla_{\bg}^2\rho$, we immediately get that the boundary has to be totally umbilic that is there exists a smooth function on $\Sigma$ such that $\nabla_{\bg}^2\rho=f\gamma$. It is crucial for our approach to note that this property is invariant under a conformal change of the metric $\bg$.

%%%%%%%%%%%%%%%%%%%%%%%%%%%%%%%%%%%%%%%%%%%%%%%%%%%%%%%%%%%%%%%%

\section{The Riemannian mapping problem of Cherrier-Escobar}\label{CherrierEscobar}

%%%%%%%%%%%%%%%%%%%%%%%%%%%%%%%%%%%%%%%%%%%%%%%%%%%%%%%%%%%%%%%%

In this section, we recall the problem of the existence on a given compact Riemannian manifold $(\overline{X}^{n},\bg)$ with boundary of a metric conformally related to $g$ with zero scalar curvature and constant mean curvature? This problem has been introduced and solve in many cases by Escobar \cite{E} although it was also addressed by Cherrier \cite{Ch} in a more general manner. In fact, Escobar proved the existence of such a metric when 
\begin{itemize}
\item $n=3,4$;
\item $n=5$ and the boundary is umbilic; 
\item $n\geq 6$ and the boundary is not umbilic;
\item $(\overline{X}^n,g)$ is locally conformally flat and the boundary is umblic. 
\end{itemize}
Later on, Marques \cite{M1,M2}, Almaraz \cite{A} and Chen \cite{C} solve the problem in many situations left by Escobar and in addition to this, Chen reduces the remaining cases to the positivity of the ADM mass of some class of asymptotically flat Riemannian manifolds. However, this positive mass theorem is not known to hold. However, very recently, Mayer and Ndiaye \cite{MN} proved all these cases left open by using the algebraic topological argument of Bahri-Coron \cite{BC} and then were allowed to settle the Riemannian mapping theorem of Cherrier and Escobar, namely
\begin{theorem}\label{CherrierEscobar}
Every $n$-dimensional compact Riemannian manifold with boundary, $n\geq 3$, and finite relative Yamabe invariant, carries a conformal scalar flat Riemannian metric with constant mean curvature on the boundary. 
\end{theorem}

Let us now briefly recall the analytic aspect of this problem and give the precise definition of the {\it relative Yamabe invariant}. For $n\geq 3$, consider a metric $\wbg=u^{\frac{4}{n-2}}\bg\in[\bg]$ where $[\bg]$ is the conformal class of $\bg$ and $u$ is a smooth positive function on $\overline{X}$. Since the scalar curvature $R_{\wbg}$ and the (normalized) mean curvature $H_{\wbg}$ are given in terms of the metric $\bg$ by
\begin{eqnarray*}
R_{\wbg}=u^{-\frac{n+2}{n-2}}L_{\bg} u\quad\text{and}\quad H_{\wbg}=u^{-\frac{n}{n-2}}B_{\bg} u
\end{eqnarray*} 
where 
\begin{eqnarray*}
L_{\bg}u=-\frac{4(n-1)}{n-2}\Delta_{\bg}u+R_{\bg}u\quad\text{and}\quad B_{\bg}=\frac{2}{n-2}\frac{\partial u}{\partial N}+H_{\bg}u.
\end{eqnarray*}
The operator $L_{\bg}$ is the conformal Laplacian of $(\overline{X},\bg)$, $\Delta_{\bg}$ is the Laplace-Beltrami operator of $(\overline{X},\bg)$ and $B_{\bg}$ is the conformal Neumann operator, $N$ being the outward unit normal of $\Sigma:=\partial\overline{X}$ in $\overline{X}$ with respect to $\bg$. Then finding a metric $\wbg\in[\bg]$ which is scalar flat and which has constant mean curvature $\mu\in\mathbb{R}$ is equivalent to find a smooth positive solution $u$ on $\overline{X}$ satisfying the elliptic boundary value problem
$$\left\lbrace
\begin{array}{ll}
L_{\bg}u = 0 & \text{ on } \overline{X}\\
B_{\bg}u = \mu u^{\frac{n}{n-2}} & \text{ on }\Sigma.
\end{array}
\right.$$
The existence of solutions of this problem is equivalent to find critical points of the Escobar functional 
\begin{eqnarray*}
\mathcal{E}_{\bg}(f):=\frac{\int_{\overline{X}}\Big(4\frac{n-1}{n-2}|\nabla^{\bg}f|^2+R_{\bg}f^2\Big)\,dv_{bg}+2(n-1)\int_{\Sigma}H_{\bg}f^2\,ds_{\bg}}{\Big(\int_{\Sigma}f^{\frac{2(n-1)}{n-2}}\,ds_{\bg}\Big)^{\frac{n-2}{n-1}}}
\end{eqnarray*}
on the space
\begin{eqnarray*}
W_+^{1,2}(\overline{X}):=\{u\in W^{1,2}(\overline{X})\,/\,u>0\}
\end{eqnarray*}
where $W^{1,2}(\overline{X})$ denotes the usual Sobolev space of functions which are $L^2$-integrable with their first derivatives. Here $dv_{\bg}$ and $ds_{\bg}$ denote respectively the Riemannian volume elements of $(\overline{X},\bg)$ and $(\Sigma,\bg)$. In fact, if the relative Yamabe invariant 
\begin{eqnarray*}
Y(\overline{X},\Sigma,[\bg]):=\inf_{f\in W^{1,2}_+(\overline{X})}\,\mathcal{E}_{\bg}(f)
\end{eqnarray*}
is finite and if $u$ denotes a positive smooth minimizer $u$ of $\mathcal{E}_{\bg}$ which can be chosen such that
\begin{eqnarray*}
\int_{\Sigma}u^{\frac{2(n-1)}{(n-2)}}\,ds_{\bg}=1
\end{eqnarray*} 
then the well-defined Riemannian metric $\wbg=u^{\frac{4}{n-2}}\bg$ satisfies
\begin{eqnarray*}
R_{\wbg}=0,\quad H_{\wbg}=\frac{1}{2(n-1)}Y(\overline{X},\Sigma,[\bg])\quad\text{and}\quad Vol(\Sigma,\wbg)=1.
\end{eqnarray*}
Note that the relative invariant is a conformal invariant of $(\overline{X}^n,\bg)$ and as observed in \cite{E2} it can be infinite. In \cite{E}, Escobar showed that
\begin{eqnarray}\label{EscIn}
Y(\overline{X},\Sigma,[\bg])\leq Y(\overline{B}^n,\mathbb{S}^{n-1},[g_0])=2(n-1)\omega_{n-1}^{\frac{1}{n-1}}
\end{eqnarray}
where $\overline{B}^n$ denotes the Euclidean ball endowed with the flat metric $g_0$ and $\omega_{n-1}$ is the volume of $(\mathbb{S}^{n-1},g_0)$. It is now immediate to prove
\begin{proposition}\label{AdaptedMetric}
Let $(X^n,g_+)$ be a Poincar\'e-Einstein manifold of class $C^2$ with a smooth representative in its conformal infinity $(\Sigma,[\gamma])$. Assume also that the Yamabe invariant $Y(\Sigma,[\gamma])$ is non-negative. Then there is a conformal compactification $\overline{g}=\rho^2g_+$, at least $C^{3,\alpha}$ up to the boundary with $\alpha\in]0,1[$ such that
\begin{eqnarray}\label{ScalFlat1}
R_{\bg}=0,\quad H_{\bg}=\frac{1}{2(n-1)}Y(\overline{X},\Sigma,[\bg])\quad\text{and}\quad Vol(\Sigma,\bg)=1.
\end{eqnarray}
\end{proposition}

\begin{proof}
The first thing to note is that under our assumptions there exists a conformal metric $\bg_1\in[g_+]$ such that $(\overline{X},\bg_1)$ is a smooth compact Riemannian manifold with boundary such that $R_{\bg_1}\geq 0$ and $H_{\bg_1}\geq 0$ (see Theorem $4$ in \cite{HM} for more details). From \cite{E}, this fact is equivalent to the non negativity of $Y(\overline{X},\Sigma,[\bg_1])$. In fact, as noticed in \cite{HM}, if $Y(\Sigma,[\gamma])>0$ we have $Y(\overline{X},\Sigma,[\bg_1])>0$. In any case, the relative Yamabe invariant of $\overline{X}$ with respect to the conformal class of $\bg_1$ is finite so that Theorem \ref{CherrierEscobar} ensures the existence of a metric $\bg\in[\bg_1]$ satisfying (\ref{ScalFlat1}). 
\end{proof}

\begin{remark}
\begin{enumerate}
\item It is obvious to see that Proposition \ref{AdaptedMetric} holds if we replace the assumption of the non-negativity of $Y(\Sigma,[\gamma])$ by the existence of a conformal compactification of $X$ with finite relative Yamabe invariant. However since our main result (see Theorem \ref{main}) is significant only when $Y(\Sigma,[\gamma])\geq 0$, we will always prefer this hypothesis. 

\item As mentioned in \cite{GH}, the regularity statement follows from a result of Chru\'sciel, Delay, Lee and Skinner \cite{CDLS}. 
\end{enumerate}
\end{remark}

%%%%%%%%%%%%%%%%%%%%%%%%%%%%%%%%%%%%%%%%%%%%%%%%%%%%%%%%%%%%%%%%

\section{Proof of Theorem \ref{main}}

%%%%%%%%%%%%%%%%%%%%%%%%%%%%%%%%%%%%%%%%%%%%%%%%%%%%%%%%%%%%%%%%

The proof is inspired by the method of Gursky-Han \cite{GH} initially introduced by Obata \cite{O} to study Yamabe metrics for Einstein manifolds. The main difference is that in our case the boundary is not totally geodesic but is only totally umbilic with constant mean curvature. Indeed, from Proposition \ref{AdaptedMetric}, let $\bg=\rho^2 g_+$ be the compactification of $X$ satisfying (\ref{ScalFlat1}). We now repeat the calculations in \cite{GH}. Indeed, if $E_{\bg}$ denotes the Einstein tensor of $(\overline{X},\bg)$, we get from the fact that $(X,g_+)$ is Einstein that
\begin{eqnarray}\label{Einst1}
E_{\bg}=-(n-2)\rho^{-1}\big(\nabla_{\bg}^2\rho-\frac{1}{n}(\Delta_{\bg}\rho)\bg\big).
\end{eqnarray}
where $\nabla_{\bg}^2$ is the Hessian operator of $(\overline{X},\bg)$
Now for $\ve>0$ small enough, we define
\begin{eqnarray*}
X_\ve:=\{x\in X\,/\,d_{\bg}(x,\Sigma)\geq\ve\}
\end{eqnarray*} 
where $d_{\bg}$ is te distance to the boundary with respect to $\bg$. Multiplying (\ref{Einst1}) by $\rho$ then taking the scalar product with $E_{\bg}$ and finally integrating on $X_\ve$ give
\begin{eqnarray*}
\int_{X_\ve}|E_{\bg}|^2_{\bg}\rho\,dv_{\bg}=-(n-2)\int_{X_\ve}\bg^{ik}\bg^{jl}(E_{\bg})_{ij}(\nabla^{\bg}_k\nabla^{\bg}_l\rho)\,dv_{\bg}.
\end{eqnarray*}
Here we also used the fact that the Einstein tensor is trace-free and $\nabla^{\bg}$ denotes the covariant derivative with respect to $\bg$. Now, we integrate by parts the right-hand side of this expression and we deduce that 
\begin{eqnarray*}
\int_{X_\ve}|E_{\bg}|^2_{\bg}\rho\,dv_{\bg}=-(n-2)\int_{\partial X_\ve}\bg^{jl}(E_{\bg})_{ij}(\nabla^{\bg}_l\rho) N^i_\ve\,ds_{\bg}(\ve),
\end{eqnarray*}
where $N_\ve$ is the outward unit normal and $ds_{\bg}(\ve)$ is the area form on $\partial X_\ve$ with respect to $\bg$. Note that the interior term in the integration by parts is zero since from the contracted second Bianchi identity 
we have
\begin{eqnarray*}
\bg^{ik}\nabla^{\bg}_k(E_{\bg})_{ij}=\frac{n-2}{2n}\nabla^{\bg}_jR_{\bg}=0
\end{eqnarray*} 
because $R_{\bg}=0$. Now once again from (\ref{Einst1}) we rewrite the previous equality as
\begin{eqnarray*}
\int_{E_\ve}|E_{\bg}|^2_{\bg}\rho\,dv_{\bg}=-(n-2)\int_{X_\ve}E_{\bg}(\nabla^{\bg}\rho,N_\ve)\,ds_{\bg}(\ve).
\end{eqnarray*}

Now as noticed in \cite{GH}, the fact that $g_+$ has constant scalar curvature implies that this metric is the unique solution of the Loewner-Nirenberg problem on $(\overline{X},\Sigma,\bg)$ (see \cite{LN}) that is it is the unique complete metric of constant scalar curvature defined in $X$ which is conformal to $\bg$. It turns out this solution has a particular asymptotic expansion near the boundary.  For example, it is proved in \cite{G} that 
\begin{eqnarray*}
\rho(x)=r-\frac{H_{\bg}}{2}r^2+\frac{R_{\gamma}}{6(n-2)}r^3+O(r^{3+\alpha})
\end{eqnarray*}
where $R_\gamma$ is the scalar curvature of $\Sigma$ for the metric $\gamma=\bg_{|\Sigma}$. In fact we only need the fact that $\rho(x)=r+O(r^2)$ and that $\nabla^{\bg}\rho=\frac{\partial}{\partial r}+O(r)$ to get
\begin{eqnarray*}
\int_{X_\ve}|E_{\bg}|^2_{\bg}\rho\,dv_{\bg}=(n-2)\int_{\partial X_\ve}E_{\bg}(N_\ve,N_\ve)\,ds_{\bg}(\ve)+O(\ve)
\end{eqnarray*} 
since $N_{\ve}=-\frac{\partial}{\partial r}$ on $\partial X_{\ve}$. Finally letting $\ve\rightarrow 0$ leads to 
\begin{eqnarray}\label{NiceForm}
\int_{X}|E_{\bg}|^2_{\bg}\rho\,dv_{\bg}=(n-2)\int_{\Sigma}E_{\bg}(N,N)\,ds_{\bg}.
\end{eqnarray}
Now from the properties (\ref{ScalFlat1}) of the metric $\bg$ and since the boundary is totally umbilic, the Gau\ss{} formula reads as
\begin{eqnarray}\label{ENN}
E_{\bg}(N,N) & = & \frac{1}{2}\big((n-1)(n-2)H_{\bg}^2-R_\gamma\big)\\
& = & \frac{1}{2}\Big(\frac{n-2}{4(n-1)}Y(\overline{X},\Sigma,[\bg])^2-R_\gamma\Big).
\end{eqnarray}
Using this expression in (\ref{NiceForm}) gives
\begin{eqnarray}\label{NF2}
\frac{(n-2)^2}{8(n-1)}Y(\overline{X},\Sigma,[\bg])^2=\int_{X}|E_{\bg}|^2_{\bg}\rho\,dv_{\bg}+\frac{(n-2)}{2}\int_\Sigma R_\gamma\,ds_{\gamma}
\end{eqnarray}
because $Vol(\Sigma,\gamma)=1$ and since the Yamabe invariant of $(\Sigma,\gamma)$ is defined by (see \cite{LP} for example)
\begin{eqnarray*}
Y(\Sigma,[\gamma]):=\inf_{\overline{\gamma}\in[\gamma]}\frac{\int_\Sigma R_{\overline{\gamma}}\,ds_{\overline{\gamma}}}{Vol(\Sigma,\overline{\gamma})^{\frac{n-3}{n-1}}}.
\end{eqnarray*}
we obtain
\begin{eqnarray*}
Y(\overline{X},\Sigma,[\bg])^2\geq 4\frac{n-1}{n-2}Y(\Sigma,[\gamma]).
\end{eqnarray*}
The expression for $n=3$ just follows from the fact that the Yamabe invariant of $\Sigma$ is exactly $4\pi\chi(\Sigma)$. 

If now we assume that equality is achieved then it is clear from (\ref{NF2}) that the metric $\bg=\rho^2 g_+$ is Ricci-flat and from (\ref{ENN}) that the boundary is totally umbilic with constant mean curvature $H_{\bg}=1$ (up to a rescaling of the metric). Moreover we also have
\begin{eqnarray*}
\nabla_{\bg}^2\rho=\frac{1}{n}(\Delta_{\bg}\rho)\bg
\end{eqnarray*}
because of (\ref{Einst1}). From the Ricci identity and the fact that $\bg$ is Ricci flat we deduce that $\nabla^{\bg}(\Delta_{\bg}\rho)=\frac{1}{n}\nabla^{\bg}(\Delta_{\bg}\rho)$ which implies since $n\geq 3$ that $\Delta_{\bg}\rho$ is a non zero constant. Indeed if $\Delta_{\bg}\rho=0$, $\rho$ has to vanish on the whole of $\overline{X}$ since $\rho_{|\Sigma}=0$ and this contadicts the fact that $\rho$ is a defining function for the boundary $\Sigma$. Then we can argue as in \cite[p.2641]{WX} to conclude that $(X^n,\bg)$ is isometric to the unit ball of $\mathbb{R}^n$ and so $(X^n,g_+)$ has to be the hyperbolic space. The converse is obvious.

%%%%%%%%%%%%%%%%%%%%%%%%%%%%%%%%%%%%%%%%%%%%%%%%%%%%%%%%%%%%%%

\section{Proof of Corollary \ref{RigHyp}}

%%%%%%%%%%%%%%%%%%%%%%%%%%%%%%%%%%%%%%%%%%%%%%%%%%%%%%%%%%%%%%

If $(X^n,g_+)$ is a Poincar\'e-Einstein manifold with the unit sphere as its conformal infinity, we have (see \cite{LP} for example)
\begin{eqnarray*}
Y(\Sigma,[\gamma])=Y(\mathbb{S}^{n-1},[g_0])=(n-1)(n-2)\omega_{n-1}^{\frac{2}{n-1}}
\end{eqnarray*}
where $[g_0]$ is the conformal class of the standard metric on the unit sphere. Now from our main result and the Escobar's inequality (\ref{EscIn}) we deduce that
\begin{eqnarray*}
4\frac{n-1}{n-2}Y(\Sigma,[\gamma])=4(n-1)^2\omega_{n-1}^{\frac{2}{n-1}}\leq Y(\overline{X},\Sigma,[\bg])^2\leq 4(n-1)^2\omega_{n-1}^{\frac{2}{n-1}}
\end{eqnarray*}
so that equality is achieved in Theorem \ref{main} and then we conclude that $(X^n,g_+)$ has to be the hyperbolic space. 
\qed

%%%%%%%%%%%%%%%%%%%%%%%%%%%%%%%%%%%%%%%%%%%%%%%%%%%%%%%%%%%%%%%%

\end{document}